\documentclass[12pt]{amsart}
\usepackage{amsmath, amsthm, amscd, amsfonts, amssymb, graphicx, color}
\usepackage{hyperref}
\usepackage{float}
\usepackage[capitalize,nameinlink]{cleveref}
\newtheorem{thm}{Theorem}
\newtheorem{rmk}{Remark}
\newtheorem{theorem}{Theorem}[section]

\newtheorem{corollary}[theorem]{Corollary}
\newtheorem{definition}[theorem]{Definition}
\theoremstyle{definition}

\newtheorem{remark}[theorem]{Remark}
\numberwithin{equation}{section}
\showoutput
\begin{document}
\setcounter{page}{1}

\title[Ricci solitons as critical points of quadratic curvature functionals]{Ricci solitons as critical points of quadratic curvature functionals}

\author[Atreyee Bhattacharya and Sayoojya Prakash]{Atreyee Bhattacharya and Sayoojya Prakash}
\address{Department of Mathematics, Indian Institute of Science Education and Research Bhopal}
\email{atreyee@math.iiserb.ac.in}
\address{Department of Mathematics, Indian Institute of Science Education and Research Bhopal}
\email{sayoojya21@iiserb.ac.in}
\begin{abstract}
Rigidity, stability and local minimizing properties of Einstein metrics as critical points of quadratic Riemannian functionals defined by $L^2$-norms of Ricci curvature, scalar curvature, Weyl curvature and Riemannian curvature have been extensively studied. However, there are non-Einstein critical points of these functionals that are not so well understood. In this paper, we study Ricci solitons, a generalization of Einstein metrics, that are critical points of a special quadratic curvature functional and analyze their rigidity.      
\end{abstract}
\maketitle
\section{Introduction}
A Riemannian manifold is said to be Einstein if its Ricci tensor is a constant multiple of the Riemannian metric tensor. A natural generalization of an Einstein manifold is a \textit{Ricci soliton}: A Riemannian manifold $(M,g)$ whose Ricci tensor satisfies the following identity: 
$$Ric_{g} + \frac{1}{2}\mathcal{L}_{X}g = \lambda g,$$ where $X \in \chi(M)$ and $\mathcal{L}_{X}g$ is the Lie derivative of $g$ with respect to $g$. Moreover, if $X$ is a gradient vector field, $(M,g)$ is said to be a \textit{gradient Ricci soliton.}  A Ricci soliton is said to be \textit{rigid} if it reduces to an Einstein manifold. The rigidity of Ricci solitons has been a topic of interest in both Riemannian geometry and theoretical Physics for the past three decades.

R. Hamilton (\cite{RHS}) proved that closed two-dimensional Ricci solitons are rigid. Following this, T. Ivey (\cite{RSTV}) established the rigidity of all three-dimensional closed Ricci solitons. Furthermore, he showed that in all dimensions, closed Ricci solitons with $\lambda \leq 0$ are rigid. G. Perelman (\cite{PRF}) proved that all closed Ricci solitons are in fact gradient Ricci solitons. P. Petersen and W. Wylie obtained multiple rigidity results for Ricci solitons in a series of papers. They first showed that a closed \textit{gradient} Ricci soliton is rigid if $\int_{M}Ric_{g}(\nabla f, \nabla f)dv_{g} \leq 0$ (\cite{RGRS}). Later, they proved (\cite{GRSPW}) that all homogeneous gradient Ricci solitons are rigid. Moreover, they showed (\cite{CGSPW}) that the only simply connected \textit{shrinking} gradient solitons with vanishing Weyl tensor are $\mathbb{S}^{n}$, $\mathbb{S}^{n-1} \times \mathbb{R}$, and $\mathbb{R}^{n}$. A. Naber (\cite{RSAN}) proved that four-dimensional non-compact Ricci solitons with a bounded non-negative curvature operator are isometric to $\mathbb{R}^{4}$ or finite quotients of $\mathbb{S}^{2} \times \mathbb{R}^{2}$, and $\mathbb{S}^{3} \times \mathbb{R}$. However, the rigidity problem for Ricci solitons with $\lambda >0$ is yet to be resolved. Recently, the authors established rigidity of certain families of closed Ricci solitons with $\lambda >0$, using conformal submersion (\cite{ABCS}) and Ricci flow (\cite{ABRF}) techniques.

In this paper, we study Ricci solitons that are critical points of a special quadratic curvature functional and analyze their rigidity. The critical points of quadratic curvature functionals and their stability form an active area of research in Riemannian geometry and geometric analysis. Let $M$ be a closed smooth manifold of dimension at least $4$. Consider the quadratic curvature functional $\mathcal{F}_{t}$, defined on the space of Riemannian metrics on $M$ by
\begin{equation*}
    \mathcal{F}_{t}(g) = \int_{M}|Ric_{g}|^{2}dv_{g} + t\int_{M}s_{g}^{2}dv_{g}
\end{equation*}
 where $Ric_{g}$ and $s_{g}$ are the Ricci curvature and the scalar curvature of $(M,g)$, respectively. We would like to understand the rigidity of Ricci solitons that are critical points of $\mathcal{F}_{t}.$
 
 It is well known that Einstein metrics are critical points of $\mathcal{F}_{t}$ restricted to the space $\mathcal{M}_1$ of Riemannian metrics on $M$ of unit volume. In \cite{GVCM}, M. Gursky and J. Viaclovsky investigated the stability of Einstein metrics as critical points of $\mathcal{F}_{t}|_{\mathcal{M}_{1}}$. The Riemannian products of Einstein manifolds $(M_{0},g_{0})$ and $(M_{1},g_{1})$ with Einstein constants $\lambda$ and $-\lambda$ are also critical points of $\mathcal{F}_{t}|_{\mathcal{M}_{1}}$. In \cite{ASSQ}, A. Bhattacharya and S. Maity explored their stability. Notably, there exist many more non-Einstein critical points of $\mathcal{F}_{t}$. M. Brozos-Vázquez, S. Caeiro-Oliveira, and E.García Río demonstrated the existence of non-Einstein cones, which are critical for all quadratic curvature functionals (\cite{MCMQ}). G. Calvaruso and A. Zaein proved that there are four-dimensional compact manifolds admitting non-Einstein Lorentzian metrics, which serve as critical points for all quadratic curvature functionals (\cite{GCSC}).
 
 In view of this, it is important to study conditions under which a critical point of $\mathcal{F}_{t}|_{\mathcal{M}_{1}}$ reduces to an Einstein metric. Significant work has already been done in this direction. G. Catino (\cite{GCNSE}) proved that every critical metric of $\mathcal{F}_{t}|_{\mathcal{M}_{1}}$ with non-negative sectional curvature must be Einstein. G. Huang, Y. Chen and X. Li showed that a critical point of $\mathcal{F}_{t}|_{\mathcal{M}_{1}}$ satisfying specific inequalities involving the norms of various curvatures reduces to an Einstein metric (\cite{GCLRP}). M. Bernardini (\cite{MBRCQ}) proved that all critical points of $\mathcal{F}_{t}|_{\mathcal{M}_{1}}$ satisfying certain sectional curvature pinching conditions also reduce to an Einstein metric.
 
 Since Ricci solitons generalize Einstein manifolds, it is important to investigate if there are non-trivial Ricci solitons that are critical points of $\mathcal{F}_{t}|_{\mathcal{M}_{1}}$. To this end, we prove the following results that also include quasi Einstein manifolds (see Definition \ref{defn:Einstein}) that are critical points of $\mathcal{F}_{t}|_{\mathcal{M}_{1}}$.
 \begin{thm}\label{TFMT}
        \begin{enumerate}
         \item Let $(M^{4},g,f,m)$ be a closed quasi-Einstein manifold. If $tr(\nabla(\mathcal{F}_{t}|_{\mathcal{M}_{1}})(g)) =0$ i.e., the point-wise trace of the gradient vector $\nabla(\mathcal{F}_{t}|_{\mathcal{M}_{1}})(g)$ is zero at all points of $M$ and $t \neq -\frac{1}{3}$, then $g$ is rigid.
     \item   Let $(M^{n},g,f)$ be a closed Ricci soliton of dimension $n\geq 5$ satisfying $tr(\nabla(\mathcal{F}_{t}|_{\mathcal{M}_{1}})(g)) =0$ as above. If $t \in (-\infty, -\frac{3n -4}{8(n-1)}] \cup [0,\infty)$, then $g$ is rigid.
 \end{enumerate}
   \end{thm}
   The following is an immediate consequence of \cref{TFMT}.
\begin{corollary}
 \begin{enumerate}
  \item Let $(M^{4},g,f,m)$ be a closed quasi-Einstein manifold such that $tr(\nabla(\mathcal{F}_{t}|_{\mathcal{M}_{1}})(g))=0$. If $t \neq -\frac{1}{3}$, then $g$ is a critical point of $\mathcal{F}_{t}|_{\mathcal{M}_{1}}$.
\item   Let $(M^{n},g,f)$ be a closed Ricci soliton of dimension $n\geq 5$ with $tr(\nabla( \mathcal{F}_{t}|_{\mathcal{M}_{1}})(g))=0$. If $t \in (-\infty, -\frac{3n -4}{8(n-1)}] \cup [0,\infty)$, then $g$ is a critical point of $\mathcal{F}_{t}(g)|_{\mathcal{M}_{1}}$.
\end{enumerate}
\end{corollary}
\begin{rmk}
It is worth mentioning that it is not always the case that if the point-wise trace of $\nabla (\mathcal{F}_{t}|_{\mathcal{M}_{1}})(g) = 0$ vanishes at all points, then $g$ must be a critical point of $\mathcal{F}_{t}|_{\mathcal{M}_{1}}$. For instance, consider closed Einstein manifolds $(M_{1},g_{1})$ and $(M_{2},g_{2})$ with Einstein constants $\lambda_{1}$ and $\lambda_{2}$, respectively, where $|\lambda_{1}| \neq |\lambda_{2}|$. Considering the Riemannian product $(M,g) = (M_{1}\times M_{2}, g_{1} + g_{2})$, it can be checked that $tr(\nabla (\mathcal{F}_{t}|_{\mathcal{M}_{1}})(g)) = 0$ at all points, but  $g$ is not a critical point of $\mathcal{F}_{t}|_{\mathcal{M}_{1}}$.
\end{rmk}
In particular, from \cref{TFMT}, we conclude the following.
\begin{thm}\label{MT}
\begin{enumerate}
    \item A closed Ricci soliton $(M^{n},g,f)$, which is a critical point of $\mathcal{F}_{t}|_{\mathcal{M}_{1}}$ is rigid, if $n = 4$ or $t \in (-\infty, -\frac{3n -4}{8(n-1)}] \cup [0,\infty)$.
    \item A closed quasi-Einstein manifold $(M^{4},g,f,m)$ that is a critical point of $\mathcal{F}_{t}|_{\mathcal{M}_{1}}$ is rigid, if $t \neq - \frac{1}{3}$.
\end{enumerate}
\end{thm}
\begin{rmk} The bound on $`t'$ in Theorems \ref{TFMT} and \ref{MT} can be improved in two ways as follows.
\begin{enumerate}
    \item (A dimension independent universal bound). It turns out that for $t \in (-\infty,-\frac{3}{8}] \cup [0,\infty)$, the rigidity results stated in Theorems \ref{TFMT} and \ref{MT} follow. We refer to \cref{DIRC} for further details.
    \item  (An improved bound depending on $g$) Using more detailed information about the scalar curvature of $(M,g)$, the bound on $`t'$ can be improved. However, the improved bound no longer remains universal. In particular, a non-rigid Ricci soliton $(M,g,f)$ is not a critical point of $\mathcal{F}_{t}|_{\mathcal{M}_{1}}$ for $t \in (-\infty, -\frac{3n-4}{8(n-1)}]\cup [k,\infty)$ where $k \geq max\{\frac{\epsilon - n}{4(n-1)},-\frac{\epsilon C}{(n-4)(M^{2} - m^{2})}\}$ with $0 < \epsilon < n$, where $\Delta s_{g}(p) = C$ where $p$ is the point where $s_{g}$ attains global minimum, $s_{min} = m$ and $s_{max} = M$.
\end{enumerate}
\end{rmk}
As an immediate consequence, we observe the following.
\begin{rmk}
    A non-trivial Ricci soliton $(M,g,f)$ with $\Delta s_{g}(p) = C$ where $p$ is the point where $s_{g}$ attains global minimum, $s_{min} = m$, $s_{max} = M$ is not a critical point of $\mathcal{F}_{t}|_{\mathcal{M}_{1}}$ for $t \in (-\infty, -\frac{3n-4}{8(n-1)}]\cup [\frac{\epsilon - n}{4(n-1)},\infty) $ for $0< \epsilon < n$ if $\frac{4(n-1)C + (n-4)(M^{2} - m^{2})}{M^{2} - m^{2}} \geq \frac{n(n-4)}{\epsilon}$.
\end{rmk}
\section{Preliminaries}
In this section, we recall some of the basic Riemannian geometric concepts and introduce the notation used throughout the paper. Let $(M^n,g)$ denote a closed Riemannian manifold of dimension $n$. Let $|\cdot|$ and $\| \cdot\|$ denote the pointwise and the global integral norms of a tensor, respectively. Likewise, $\langle \cdot \rangle$ and $\langle \langle \cdot \rangle \rangle$ denote the pointwise and the global integral inner products of a tensor induced by $g$, respectively.
\begin{definition}
 The \textbf{Riemann curvature tensor} of a Riemannian manifold $(M,g)$ is given by,
        \begin{equation}
            R(X,Y,Z,W) = g(\nabla_{Y}\nabla_{X}Z - \nabla_{X}\nabla_{Y}Z + \nabla_{[X,Y]}Z,W)
        \end{equation}
        where $X,Y,Z$, and $W$ are arbitrary vector fields in an open set of $M$ and $\nabla$ is the Levi-Civita connection of $(M,g)$.
\end{definition}
\noindent Let $\mathcal{M}$ denote the space of all Riemannian metrics on $M$. Then,
\begin{equation*}
    T_{g}\mathcal{M} = \mathcal{S}^{2}(M)
\end{equation*}
where $\mathcal{S}^{2}(M)$ is the space of all symmetric $2$ tensor fields on $M$.\\
\noindent Let $\mathcal{M}_{1}$ denote the space of all Riemannian metrics on $M$ with unit volume. Then,
\begin{equation*}
    T_{g}\mathcal{M}_{1}\,=\, \{h \in \mathcal{S}^{2}(M)\,|\,\langle \langle h,g\rangle \rangle = \int_{M}tr(h)dv_{g}=0\}
\end{equation*}
where $tr(h)$ denotes the point-wise trace of $h$ with respect to $g$.
\begin{definition}
    A real-valued function $F : \mathcal{M} \rightarrow \mathbb{R}$ is called a \textbf{Riemannian functional} if $F(\phi^{*}g)= F(g)$ for every diffeomorphism $\phi$ on $M$ and every $g \in \mathcal{M}$.
\end{definition}
\begin{definition}
    A symmetric $2$-tensor field $\nabla F_{g}$ is called the \textbf{gradient of $F$ at $g$}  if
    \begin{equation*}
        \langle \langle \nabla F_{g} ,h\rangle \rangle = \frac{d}{dt}F(g+th)|_{t=0}
    \end{equation*}
    for every $h \in \mathcal{S}^{2}(M)$. A metric $g \in \mathcal{M}_{1}$ is called a \textbf{critical point} of a functional $F$ if $\nabla F_{g}$ vanishes along $T_{g}\mathcal{M}_{1}$.
\end{definition}
We consider the following quadratic curvature functionals in this paper defined by the $L^2$-norms of Ricci  and scalar curvature, as follows
\begin{equation*}
    \mathcal{S}(g) = \int_{M}s_{g}^{2}dv_{g}
\end{equation*}
\begin{equation*}
    \mathcal{R}ic(g) = \int_{M}|Ric_{g}|^{2}dv_{g}
\end{equation*}
\begin{equation*}
    \mathcal{F}_{t}(g) = \mathcal{S}(g) + t \mathcal{R}ic(g)
\end{equation*}
where $s_{g}$ and $Ric_{g}$ denote the scalar curvature and Ricci curvature of $(M,g)$, respectively.

The gradients of the quadratic curvature functionals $\mathcal{S}$ and $\mathcal{R}ic$ restricted to the space $\mathcal{M}_1$ are given below (see \cite{BEM} for details),

\small{\begin{eqnarray}\label{E1}
 \nonumber   \nabla (\mathcal{S}|_{\mathcal{M}_{1}})(g) &=& 2Dds_{g} - 2 (\Delta s_{g})g - 2s_{g}.Ric_{g} + \frac{1}{2}s_{g}^{2}g +(\frac{2}{n} - \frac{1}{2})\|s_{g}\|^{2}g \text{ and}\\
 \nabla (\mathcal{R}ic|_{\mathcal{M}_{1}})(g) &=& D^{*}DRic_{g} + Dds_{g} -  \frac{1}{2} (\Delta s_{g})g \\
 \nonumber &&+ \frac{1}{2}|Ric_{g}|^{2}g -2R_{g}^{\circ}(Ric_{g}) +(\frac{2}{n} - \frac{1}{2})\|Ric_{g}\|^{2}g
\end{eqnarray}}
where $\Delta s_{g} = tr(Dds_{g})$ and $R_{g}^{\circ}(Ric_{g})(X,Y) = \langle R(X,-,Y,-),Ric_{g} \rangle$.
\begin{definition}\label{defn:Einstein}
    A Riemannian manifold $(M,g)$ is said to be
    \begin{enumerate}
        \item an \textbf{Einstein manifold} if $Ric_{g} = \lambda g$, where $\lambda \in \mathbb{R}$.
        \item a \textbf{Ricci soliton} if there exists a vector field $X \in \chi(M)$ such that
        \begin{equation*}
            Ric_{g} + \frac{1}{2}\mathcal{L}_{X}g = \lambda g
        \end{equation*} 
        where $\mathcal{L}_{X}g$ is the Lie derivative of $g$ with respect to $X$ and $\lambda \in \mathbb{R}$. In particular, if $X = \nabla f$, for some $f \in C^{\infty}(M)$, then $(M,g)$ is called a \textbf{gradient Ricci soliton}.
        \item a \textbf{quasi-Einstein manifold} if there exists a function $f \in C^{\infty}(M)$ such that
        \begin{equation*}
            Ric_{g} + Hess_{g}(f) - \frac{1}{m} df \otimes df = \lambda g
        \end{equation*}
        where $0 < m \leq \infty$, $\lambda \in \mathbb{R}$ and $Hess_{g}(f)(X,Y) = g(\nabla_{X}\nabla f,Y)$.
    \end{enumerate}
\end{definition}
\begin{remark} \label{rs}
\begin{enumerate}
    \item It turns out that if $m = \infty$ in the above definition, then a quasi-Einstein manifold becomes a gradient Ricci soliton.
    \item Closed Ricci solitons are gradient Ricci solitons (see \cite{PRF}).
\end{enumerate}
\end{remark}
Thus quasi-Einstein manifolds generalize the Einstein manifolds and we have the following definition.
\begin{definition}
    A quasi-Einstein manifold is said to be \textbf{rigid} if it reduces to an Einstein manifold.
\end{definition}
\begin{remark}
    A closed quasi-Einstein manifold is rigid if and only if $f$ is constant, which is not true in a non-compact setup (\cite{RQE}). Also, using the second part of Remark \ref{rs}, it follows that the rigidity of closed Ricci solitons is a particular case of the rigidity of closed quasi-Einstein manifolds.  
\end{remark}
\section{Main Results}
Note that it suffices to consider closed Ricci solitons of dimension at least $4$ as closed Ricci solitons of dimension $3$ or less are rigid. As an immediate consequence of Proposition 3.1 of \cite{CPSR}, we have the following observation for closed quasi-Einstein manifolds.
\begin{theorem}
    A closed quasi-Einstein manifold which is a critical point of the quadratic scalar curvature functional $\mathcal{S}$ restricted to $\mathcal{M}_{1}$, is rigid.
\end{theorem}
Next, we obtain the following result for closed quasi-Einstein manifolds of dimension $4$.
\begin{theorem}\label{DFRT}
A closed quasi-Einstein manifold $(M^{4},g,f,m)$ satisfying $tr(\nabla(\mathcal{F}_{t}|_{\mathcal{M}_{1}})(g)) =0$ such that $t \neq -\frac{1}{3}$ is rigid.
\end{theorem}
\begin{proof}
  Using Equation \eqref{E1}, we have,
    \begin{align}
        \nabla (\mathcal{F}_{t}|_{\mathcal{M}_{1}})(g)=&  D^{*}DRic_{g} + Dds_{g} -  \frac{1}{2} (\Delta s_{g})g + \frac{1}{2}|Ric_{g}|^{2} -2R_{g}^{\circ}(Ric_{g}) \nonumber\\& + t (2Dds_{g} - 2 (\Delta s_{g})g - 2s_{g}.Ric_{g} + \frac{1}{2}s_{g}^{2}g).
    \end{align}
We have,
    \begin{align}
        tr(\nabla (\mathcal{F}_{t}|_{\mathcal{M}_{1}})(g)) =& 0\nonumber\\
     \text{i.e. }   (1 + 3t)\Delta s_{g} =& 0\nonumber
    \end{align}
     If $t \neq - \frac{1}{3}$, then $\Delta s_{g} = 0$. Hence $(M,g,f)$ is rigid by the maximal principle. 
\end{proof}
For a closed Ricci soliton of dimensions higher than $4$, we have the following general result.
\begin{theorem}\label{CRSTT}
A closed Ricci soliton $(M^{n},g,f)$ with $n\geq 5$, satisfying $tr(\nabla (\mathcal{F}_{t}|_{\mathcal{M}_{1}})(g)) = 0$, is rigid, for $t \in (-\infty, -\frac{3n -4}{8(n-1)}] \cup [0,\infty)$.
\end{theorem}
\begin{proof}
In this case, we have (using Equation \eqref{E1})
\begin{align}\label{FTGE}
        \nabla (\mathcal{F}_{t}|_{\mathcal{M}_{1}})(g) =&  D^{*}DRic_{g} + Dds_{g} -  \frac{1}{2} (\Delta s_{g})g + \frac{1}{2}|Ric_{g}|^{2}g -2R_{g}^{\circ}(Ric_{g}) +(\frac{2}{n} - \frac{1}{2})\|Ric_{g}\|^{2}g\nonumber\\& + t (2Dds_{g} - 2 (\Delta s_{g})g - 2s_{g}.Ric_{g} + \frac{1}{2}s_{g}^{2}g +(\frac{2}{n} - \frac{1}{2})\|s_{g}\|^{2}g).\nonumber
    \end{align}
We have,
\begin{equation*}
    tr(\nabla \mathcal{F}_{t}(g)|_{\mathcal{M}_{1}}) = 0.
\end{equation*}
\begin{align}\label{TGFTE}
-(\frac{n}{2} + 2(n-1)t)\Delta s_{g} + (\frac{n}{2} -2)(|Ric_{g}|^{2}-\|Ric_{g}\|^{2}) + t((\frac{n}{2}-2)(s_{g}^{2} - \|s_{g}\|^{2}))=0.
\end{align}
We divide proof in three cases: $t \geq 0$, $t \leq -\frac{3n-4}{4n}$ and $-\frac{3n-4}{4n} \leq t \leq - \frac{3n-4}{8(n-1)}$.\\

\noindent Case I: $t \geq 0$.\\

\noindent Let $p$ be the point where $s_{g}$ attains a global minimum. At $p$, $\Delta s_{g}(p) \geq 0$ and $s_{g}^{2}(p) \leq \|s_{g}\|^{2}$. Hence by \cref{TGFTE},
\begin{equation}
    |Ric_{g}|^{2}(p) \geq \|Ric_{g}\|^{2}
\end{equation}
\begin{align}\label{TNRE}
    |Ric_{g}|^{2}(p) \geq\,\, & \|Ric_{g}\|^{2}= \|\lambda g - Hess f\|^{2} = n \lambda^{2} + \|Hessf\|^{2}.\nonumber\\
    |Ric_{g}|^{2}(p) \geq\,\, & n \lambda^{2} + \|Hessf\|^{2} . 
\end{align}
From Lemma 2.5 of \cite{RGRS},
\begin{align}\label{PICE}
    \frac{1}{2}(\Delta s_{g} - \nabla_{\nabla f}s_{g}) =& -|Ric_{g} - \frac{s_{g}}{n}g|^{2} +s_{g}(\lambda - \frac{1}{n}s_{g}).\nonumber\\
    \frac{1}{2}(\Delta s_{g} - \nabla_{\nabla f}s_{g}) =& \lambda s_{g} - |Ric_{g}|^{2}.
\end{align}
At $p$, $\nabla s_{g}$ vanishes and $\Delta s_{g}(p) \geq 0$. Hence,
\begin{equation}\label{MSCE}
    \lambda s_{min} \geq |Ric_{g}|^{2}(p).
\end{equation}
Using \cref{TNRE} in \cref{MSCE},
\begin{equation}\label{MSE}
    \lambda s_{min} \geq n \lambda^{2} + \|Hessf\|^{2}.
\end{equation}
Since $s_{g} = n\lambda - \Delta f$, $s_{min} \leq n\lambda$. By \cref{MSE}, we will get $Hess(f) = 0$. Hence $(M,g,f)$ is rigid.\\

\noindent Case II: $t \leq - \frac{3n -4}{4n}$.\\

\noindent Multiplying \cref{TGFTE} by $s_{g}$ and integrating and using the divergence theorem gives,
\begin{align}\label{SCIE}
    (n + 4(n-1)t)\|\nabla s_{g}\|^{2} + (n-4)\int_{M}(s_{g}|Ric_{g}|^{2} + ts_{g}^{3})dv_{g} = n(n-4)\lambda (\|Ric_{g}\|^{2} + t\|s_{g}\|^{2})
\end{align}
Multiplying \cref{PICE} by $s_{g}$ and integrating and using the divergence theorem,
\begin{equation}\label{GSCE}
    \|\nabla s_{g}\|^{2} = \frac{(n-4)\lambda}{2}\|s_{g}\|^{2} - \frac{1}{2}\int_{M}s_{g}^{3}dv_{g} + 2\int_{M}s_{g}|Ric_{g}|^{2}dv_{g}
\end{equation}
Integrating \cref{PICE} and using the divergence theorem,
\begin{equation}\label{NRCE}
    \|Ric_{g}\|^{2} = \frac{1}{2}\|s_{g}\|^{2} - \frac{n(n-2)\lambda^{2}}{2}
\end{equation}
Substituting \cref{NRCE} and \cref{GSCE} in \cref{SCIE}, we get
\begin{equation}
    2(n-2)t\|\nabla s_{g}\|^{2} + (3n -4 + 4nt)\int_{M}s_{g}|Ric_{g} - \frac{s_{g}}{n}g|^{2}dv_{g} + \frac{(n-2)(n-4)}{2}(n^{2}\lambda^{3} - \frac{1}{n}\int_{M}s_{g}^{3}dv_{g}) = 0
\end{equation}
Since $\int_{M}s_{g}^{3} \geq (n\lambda)^{3}$ and $t \leq - \frac{3n -4}{4n}$, $(M,g,f)$ is rigid.\\

\noindent Case III : $-\frac{3n-4}{4n}\leq t \leq - \frac{3n-4}{8(n-1)}$.\\

\noindent We have, 
\begin{equation}\label{GSRE}
    2(n-2)t\|\nabla s_{g}\|^{2} + (3n -4 + 4nt)\int_{M}s_{g}|Ric_{g}|^{2}dv_{g} -(\frac{n}{2}+4t)\int_{M} \frac{s_{g}^{3}}{n}dv_{g} + \frac{(n-2)(n-4)}{2}n^{2}\lambda^{3} = 0
\end{equation}
Substituting \cref{GSCE} in \cref{GSRE}, 
\begin{equation}
    (\frac{3n-4}{2}+4(n-1)t)\|\nabla s_{g}\|^{2} + (n-4)(\frac{n^{2}(n-2)\lambda^{3}}{2} + (\frac{1}{4}+t)\int_{M}s_{g}^{3}dv_{g}- (\frac{3n-4}{4n} + nt)\lambda \|s_{g}\|^{2}) = 0
\end{equation}
Since $\int_{M}s_{g}^{3}dv_{g} \geq (n\lambda)^{3}$, $\|s_{g}\|^{2} \geq (n\lambda)^{2}$ and $-\frac{3n-4}{4}\leq t \leq - \frac{3n-4}{8(n-1)}$, $(M,g,f)$ is rigid.
\end{proof}
 The critical points of $\mathcal{F}_{-\frac{1}{3}}|_{\mathcal{M}_{1}}$ are Bach-flat by \cite{GVCM}. Since Bach flat Ricci solitons are rigid by \cite{BFSS}, by \cref{DFRT} and \cref{CRSTT}, we conclude \cref{MT}.\\
Since $-\frac{3}{8} < - \frac{3}{8} + \frac{1}{8(n-1)} = - \frac{3n-4}{8(n-1)}$, we get the following corollary.
\begin{corollary}\label[corollary]{DIRC}
    A Ricci soliton $(M^{n},g,f)$ that is a critical point of $\mathcal{F}_{t}|_{\mathcal{M}_{1}}$, is rigid where $t \in (-\infty, -\frac{3}{8}] \cup [0,\infty)$.
\end{corollary}
\bibliographystyle{amsplain}
\bibliography{references}
\end{document}